\documentclass[preprint,12pt]{elsarticle}

\usepackage{amssymb}
\usepackage{amsthm}

\usepackage{amsmath}

\usepackage{bm}

\usepackage{dsfont}

\usepackage{mathscinet}

\newtheorem{theorem}{Theorem}[section]
\newtheorem{lemma}[theorem]{Lemma}
\newtheorem{proposition}[theorem]{Proposition}
\newtheorem{corollary}[theorem]{Corollary}

\theoremstyle{definition}

\theoremstyle{remark}

\makeatletter
\def\@author#1{\g@addto@macro\elsauthors{\normalsize%
    \def\baselinestretch{1}%
    \upshape\authorsep#1\unskip\textsuperscript{%
      \ifx\@fnmark\@empty\else\unskip\sep\@fnmark\let\sep=,\fi
      \ifx\@corref\@empty\else\unskip\sep\@corref\let\sep=,\fi
      }%
    \def\authorsep{\unskip,\space}%
    \global\let\@fnmark\@empty
    \global\let\@corref\@empty  
    \global\let\sep\@empty}%
    \@eadauthor={#1}
}
\makeatother

\begin{document}

\begin{frontmatter}



\title{Explicit expressions for a certain class of Appell polynomials. A probabilistic approach}


\author{Jos\'{e} A. Adell\corref{a1}\fnref{thanks}}
\ead{adell@unizar.es}

\author{Alberto Lekuona\fnref{thanks}}
\ead{lekuona@unizar.es}

\cortext[a1]{Corresponding author.}
\fntext[thanks]{The authors are partially supported by Research Projects DGA (E-64), MTM2015-67006-P, and by FEDER funds.}

\address{Departamento de M\'{e}todos Estad\'{\i}sticos, Facultad de
Ciencias, Universidad de Zaragoza, 50009 Zaragoza (Spain)}

\begin{abstract}
We consider the class $\mathcal{E}_t(Y)$ of Appell polynomials whose generating function is given by means of a real power $t$ of the moment generating function of a certain random variable $Y$. For such polynomials, we obtain explicit expressions depending on the moments of $Y$. It turns out that various kinds of generalizations of Bernoulli and Apostol-Euler polynomials belong to $\mathcal{E}_t(Y)$ and can be written and investigated in a unified way. In particular, explicit expression for such polynomials can be given in terms of suitable probabilistic generalizations of the Stirling numbers of the second kind.
\end{abstract}


\end{frontmatter}


\section{Introduction and main result}\label{s1}

The Appell polynomials have many important applications in different areas of pure and applied mathematics. Paradigmatic examples of such polynomials are the classical Bernoulli and Euler polynomials, as well as their various kinds of generalizations. Good general references for this subject are Erd\'{e}lyi \textit{et al.} \cite{ErdMagObeTri1955}, Graham \textit{et al.} \cite{GraKnuPat1989}, Abramowitz and Stegun \cite{AbrSte1992}, and Olver \textit{et al.} \cite{OlvLozBoiCla2010}, among others.

An interesting problem, widely considered in the literature, is to find explicit expressions for the different examples of Appell polynomials. As far as we know, this problem has usually been carried out by taking into account the particular form of the Appell polynomials in each case. For instance, Todorov \cite{Tod1985} and Srivastava and Todorov \cite{SriTod1988} obtained an explicit formula for generalized Bernoulli polynomials of complex order in terms of the Stirling numbers of the second kind (a similar, but different formula for the Bernoulli numbers has been provided by Guo and Qi \cite{GuoQi2014}). Luo \cite{Luo2006} gave an explicit representation of Apostol-Euler polynomials of real order in terms of Gaussian hypergeometric functions. Other identities for higher order Apostol-Genocchi polynomials can be found in Luo and Srivastava \cite{LuoSri2011} (see also Wuyungaowa \cite{Wuy2013} for identities of products of Genocchi numbers).

Let $\mathds{N}$ be the set of positive integers and $\mathds{N}_0=\mathds{N}\cup \{0\}$. Unless otherwise specified, we assume from now on that $n\in \mathds{N}_0$, $x\in \mathds{R}$, and $z\in \mathds{C}$ with $|z|\leq r$, where $r>0$ may change from line to line. Denote by $\mathcal{G}$ the set of real sequences $\boldsymbol{u} =(u_n)_{n\geq 0}$ such that $u_0\neq 0$ and
\begin{equation*}
\sum_{n=0}^\infty |u_n|\dfrac{r^n}{n!}<\infty ,
\end{equation*}
for some radius $r>0$. If $\boldsymbol{u}\in \mathcal{G}$, we denote its generating function as
\begin{equation*}
G(\boldsymbol{u},z)=\sum_{n=0}^\infty u_n \dfrac{z^n}{n!}.
\end{equation*}

Let $A(x)=(A_n(x))_{n\geq 0}$ be a sequence of polynomials such that $A(0)=(A_n(0))_{n\geq 0}\in \mathcal{G}$. Recall that $A(x)$ is called an Appell sequence if one of the following equivalent conditions is satisfied
\begin{equation}\label{eq1.1}
A'_n(x)=nA_{n-1}(x),\quad n\in \mathds{N},
\end{equation}
\begin{equation}\label{eq1.2}
A_n(x)=\sum_{k=0}^n \binom{n}{k}A_k(0)x^{n-k},
\end{equation}
or
\begin{equation}\label{eq1.3}
G(A(x),z)=G(A(0),z)e^{xz}.
\end{equation}
Finally, let $Y$ be a random variable such that
\begin{equation}\label{eq1.4}
r\mathds{E}|Y|e^{r|Y|}<1,
\end{equation}
where $\mathds{E}$ stands for mathematical expectation.

In a recent paper, Ta \cite{Ta2015} has introduced the set $\mathcal{R}(Y)$ of Appell sequences $A(x)$ such that
\begin{equation}\label{eq1.5}
G(A(x),z)= \dfrac{e^{xz}}{\mathds{E}e^{zY}}.
\end{equation}
As we will see in inequality \eqref{eq2.11star} below, condition \eqref{eq1.4} implies that formula \eqref{eq1.5} is well defined. The set $\mathcal{R}(Y)$ contains the generalized Bernoulli and Euler polynomials of integer order, as well as the Hermite and the Miller-Lee polynomials, among others. For such polynomials, Ta \cite{Ta2015} gives representations in terms of moments of complex valued random variables. However, important special examples usually considered in the literature, such as the generalized Bernoulli and Apostol-Euler  polynomials of real order fall out of the scope of definition \eqref{eq1.5}.

The aim of this paper is twofold. On the one hand, to enlarge the set $\mathcal{R}(Y)$ by considering the set $\mathcal{E}_t(Y)\supseteq \mathcal{R}(Y)$ of Appell sequences $A(t;x)$ whose generating function is given by
\begin{equation}\label{eq1.6}
G(A(t;x),z)= \dfrac{e^{xz}}{(\mathds{E}e^{zY})^t},
\end{equation}
for some real $t$. We obtain a simple explicit expression of any Appell sequence $A(t;x)\in \mathcal{E}_t(Y)$ in terms of the moments of the random variable $Y$ (see Theorem~\ref{th1} and the comments following it).

On the other hand, the different kinds of generalizations of the Bernoulli and Apostol-Euler polynomials have traditionally been investigated following separate ways. Since all of these kinds of polynomials belong to the set $\mathcal{E}_t(Y)$, we show that a unified treatment is possible by introducing a probabilistic generalization of the Stirling numbers of the second kind. In fact, we obtain explicit expressions of any Appell sequence $A(t,x)\in \mathcal{E}_t(Y)$ in terms of such numbers (see Theorem~\ref{th1}).

Throughout this paper, $Y$ is a random variable satisfying \eqref{eq1.4} and $(Y_j)_{j\geq 1}$ is a sequence of independent copies of $Y$. We denote by
\begin{equation}\label{eq1.7}
S_k=Y_1+\cdots +Y_k,\quad k\in \mathds{N},\qquad S_0=0.
\end{equation}

The Stirling numbers of the second kind $S(n,m)$, $n,m\in \mathds{N}_0$, $m\leq n$, can be defined in various equivalent ways (see, for instance, Abramowitz and Stegun \cite[p. 824]{AbrSte1992}). We recall here the following definition
\begin{equation*}
    S(n,m)=\dfrac{1}{m!}\sum_{k=0}^m \binom{m}{k}(-1)^{m-k}k^n.
\end{equation*}
Different generalizations of these numbers have been considered in the literature (see, for instance, Hsu and Shiue \cite{HsuShi1998}, Mez\H{o} \cite{Mez2010}, Luo and Srivastava \cite{LuoSri2011} and Wuyungaowa \cite{Wuy2013}). Here, we consider the following probabilistic generalization
\begin{equation}\label{eq1.1star}
    S_Y(n,m)=\dfrac{1}{m!}\sum_{k=0}^m \binom{m}{k}(-1)^{m-k}\mathds{E}S_k^n,\quad n,m\in \mathds{N}_0,\quad m\leq n.
\end{equation}
Observe that if $Y=1$, then $S_1(n,m)=S(n,m)$. Some properties of $S_Y(n,m)$ are given in Section~\ref{s2}.

With these ingredients, we enunciate the main result of this paper.

\begin{theorem}\label{th1}
Let $t\in \mathds{R}$ and $A(t;x)\in \mathcal{E}_t(Y)$. Then,
\begin{equation}\label{eq1.8}
A_n(t;0)=\sum_{m=0}^n \binom{-t}{m}m! S_Y(n,m)=\sum_{k=0}^n \binom{-t}{k}\binom{n+t}{n-k} \mathds{E}S_k^n.
\end{equation}
\end{theorem}

In view of \eqref{eq1.2}, Theorem~\ref{th1} gives us two explicit expressions for any Appell sequence $A(t;x)\in \mathcal{E}_t(Y)$. The first one in terms of the numbers $S_Y(n,m)$. The second one in terms of the moments of the random variable $Y$. In fact, denote by
\begin{equation}\label{eq1.10}
\mu_j=\mathds{E}Y^j,\quad j\in \mathds{N}_0,
\end{equation}
the $j$th moment of $Y$. Observe that we have from \eqref{eq1.7}
\begin{equation}\label{eq1.11}
\mathds{E}S_k^n=\sum_{j_1+\cdots +j_k=n} \dfrac{n!}{j_1!\cdots j_k!}\, \mu_{j_1}\cdots \mu_{j_k},\qquad k,n\in \mathds{N}_0.
\end{equation}

The proof of Theorem~\ref{th1}, in which probabilistic methods play an important role, is given in Section~\ref{s3}. Applications of this theorem to different generalizations of Bernoulli and Apostol-Euler polynomials are deferred to Section~\ref{s4}.

\section{The numbers $S_Y(n,m)$}\label{s2}

Let $U$ be a random variable uniformly distributed on $[0,1]$ and independent of $Y$. Also, let $(U_j)_{j\geq 1}$ be a sequence of independent copies of $U$, which are assumed to be independent of $(Y_j)_{j\geq 1}$, as given in \eqref{eq1.7}.

On the other hand, let $m\in \mathds{N}$ and $\alpha_1,\ldots ,\alpha_m \in \mathds{R}$. Denote by
\begin{equation*}
I_m(k)=\left \{(i_1,\ldots ,i_k)\in \{1,\ldots ,m\}:\ i_r\neq i_s,\text{ if } r\neq s \right \}, \quad k=1,\ldots ,m.
\end{equation*}
If $f:\mathds{R}\to \mathds{R}$ is an arbitrary function, we define
\begin{equation}\label{eq2.14}
\Delta^m_{\alpha_1,\ldots, \alpha_m} f(x)= (-1)^mf(x)+\sum_{k=1}^m (-1)^{m-k}\sum_{I_m(k)}f(x+\alpha_{i_1}+\cdots +\alpha_{i_k}).
\end{equation}
Observe that if $\alpha_1=\cdots =\alpha_m=\alpha$, then
\begin{equation*}
\Delta_\alpha^m f(x):=\Delta^m_{\alpha ,\ldots, \alpha} f(x)=\sum_{k=0}^m \binom{m}{k}(-1)^{m-k}f(x+\alpha k)
\end{equation*}
is the $m$th difference of $f$ at step $\alpha$. It has been shown in \cite[Lemma 7.2]{AdeLek2017.2} that if $f$ is $m$ times differentiable, then
\begin{equation}\label{eq2.15}
\Delta^m_{\alpha_1,\ldots ,\alpha_m} f(x)=\alpha_1 \cdots \alpha_m \mathds{E}f^{(m)}(x+\alpha_1 U_1+\cdots +\alpha_m U_m).
\end{equation}
This implies that if $p_n(y)$ is a polynomial of exact degree $n$, then
\begin{equation}\label{eq2.6star}
\Delta_{\alpha_1,\ldots, \alpha_m}p_n(x)=0,\quad m=n+1,n+2,\ldots
\end{equation}

The generalized $m$th differences of $f$ defined in \eqref{eq2.14} are actually iterates of the difference operator
\begin{equation*}
    \Delta_\alpha^1 f(x)=f(x+\alpha)-f(x),\quad \alpha \in \mathds{R}.
\end{equation*}
Actually, it can be shown by induction that
\begin{equation*}
    \Delta_{\alpha_1,\cdots, \alpha_m}^mf(x)=\left ( \Delta_{\alpha_1}^1 \circ \cdots \circ \Delta_{\alpha_m}^1 \right )f(x),\quad m\in \mathds{N}.
\end{equation*}
Such iterates have been considered in Mrowiec \textit{et al.} \cite{MroRajWas2017} in connection with Wright-convex functions of order $m$. With a different notation, these generalized difference operators have been used in Dilcher and Vignat \cite{DilVig2016} in connection with convolution identities for Bernoulli and Euler polynomials. The following auxiliary result, which has an independent interest, will be crucial to give various equivalent definitions of the numbers $S_Y(n,m)$. We make use of the conventions
\begin{equation}\label{eq2.7star}
\Delta_\phi^0 f(x)=f(x),\quad \sum_{k=1}^0 x_k=0,\quad \prod_{k=1}^0 x_k=1.
\end{equation}

\begin{lemma}\label{l2}
Let $m\in \mathds{N}_0$. If $f$ is $m$ times differentiable, then
\begin{equation*}
\begin{split}
  \mathds{E}\Delta^m_{Y_1,\ldots ,Y_m} f(x) &= \sum_{k=0}^m \binom{m}{k}(-1)^{m-k} \mathds{E}f(x+S_k)\\
  &= \mathds{E}Y_1\cdots Y_m f^{(m)}(x+Y_1U_1+\cdots +Y_mU_m),
\end{split}
\end{equation*}
provided that the preceding expectations exist.
\end{lemma}

\begin{proof}
By \eqref{eq2.7star}, the result is true for $m=0$. We thus assume that $m\in \mathds{N}$. Let $(i_1,\ldots ,i_k)\in I_m(k)$, $k=1,\ldots ,m$. By \eqref{eq1.7} and the fact that $(Y_j)_{j\geq 1}$ is a sequence of independent copies of $Y$, the random variables $S_k$ and $Y_{i_1}+\cdots +Y_{i_k}$ have the same law, thus implying that
\begin{equation}\label{eq2.15.2}
  \mathds{E} f(x+Y_{i_1}+\cdots +Y_{i_k})=\mathds{E}f (x+S_k).
\end{equation}
Therefore, the result follows from \eqref{eq2.14} and \eqref{eq2.15}, by replacing $\alpha_{i_j}$ by $Y_{i_j}$, $j=1,\ldots ,k$, and then taking expectations.
\end{proof}

\begin{proposition}\label{prop2.2}
Let $m, n\in \mathds{N}_0$, with $m\leq n$, and $I_n(y)=y^n$, $y\in \mathds{R}$. Then,
\begin{equation}\label{eq2.8star}
\begin{split}
S_Y(n,m)&=\dfrac{1}{m!}\mathds{E}\Delta_{Y_1,\ldots ,Y_m}^m I_n(0)= \dfrac{1}{m!}\sum_{k=0}^m \binom{m}{k}(-1)^{m-k}\mathds{E}S_k^n\\
&=\binom{n}{m}\mathds{E}Y_1\cdots Y_m \left ( Y_1U_1+\cdots +Y_mU_m\right )^{n-m}.
\end{split}
\end{equation}
Equivalently, the numbers $S_Y(n,m)$ are defined via its generating function
\begin{equation}\label{eq2.9star}
    \dfrac{\left (\mathds{E}e^{zY}-1 \right )^m}{m!}=\sum_{n=m}^\infty \dfrac{S_Y(n,m)}{n!}z^n.
\end{equation}
\end{proposition}

\begin{proof}
Choosing $x=0$ and $f(y)=I_n(y)$, $y\in \mathds{R}$, in Lemma~\ref{l2}, formula \eqref{eq2.8star} follows from definition \eqref{eq1.1star} and convention \eqref{eq2.7star}. On the other hand, $S_Y(n,0)=\delta_{n0}$, thus implying that \eqref{eq2.9star} is true for $m=0$. Assume that $m\in \mathds{N}$. Observe that
\begin{equation}\label{eq2.10star}
    \mathds{E}\left ( e^{zY}-1\right )=z\mathds{E}Ye^{zUY},
\end{equation}
which implies, by virtue of \eqref{eq1.4}, that
\begin{equation}\label{eq2.11star}
    \left | \mathds{E}\left ( e^{zY}-1\right )\right |\leq |z|\mathds{E}|Y|e^{|z||Y|}\leq r \mathds{E}|Y|e^{r|Y|}<1,\quad |z|\leq r.
\end{equation}
Assume that $|z|\leq r$. Using \eqref{eq2.10star} and the independence between the random variables involved, we get
\begin{equation*}
    \begin{split}
        \dfrac{\left ( \mathds{E}e^{zY}-1\right )^m}{m!}&=\dfrac{1}{m!}\mathds{E}\left ( e^{zY_1}-1\right )\cdots \mathds{E}\left ( e^{zY_m}-1\right ) \\
         & = \dfrac{z^m}{m!}\mathds{E}Y_1\cdots Y_m e^{z(U_1Y_1+\cdots +U_mY_m)}\\
         &=\dfrac{z^m}{m!}\sum_{k=0}^\infty \dfrac{z^k}{k!}\mathds{E}Y_1\cdots Y_m (U_1Y_1+\cdots +U_mY_m)^k\\
         & = \sum_{n=m}^\infty \dfrac{z^n}{n!}\binom{n}{m}\mathds{E}Y_1\cdots Y_m (U_1Y_1+\cdots +U_mY_m)^{n-m}.
     \end{split}
\end{equation*}
This, together with \eqref{eq2.8star}, shows \eqref{eq2.9star} and completes the proof.
\end{proof}

Setting $Y=1$ in \eqref{eq2.9star}, we obtain the well known generating function for the Stirling numbers of the second kind $S(n,m)$. In the same way, we have from the last equality in \eqref{eq2.8star}
\begin{equation}\label{eq2.12star}
    S(n,m)=\binom{n}{m}\mathds{E}(U_1+\cdots +U_m)^{n-m}.
\end{equation}
This probabilistic representation was already shown by Sun \cite{Sun2005} (see also \cite[formula (38)]{GraKnuPat1989}).

\section{Proof of Theorem~\ref{th1}}\label{s3}

Let $t\in \mathds{R}$ and $|z|\leq r$, as in \eqref{eq2.11star}. By \eqref{eq2.11star}, we can apply the binomial expansion and formula \eqref{eq2.9star} to obtain
\begin{equation*}
    \begin{split}
       & (\mathds{E}e^{zY})^{-t}=\left ( 1+\mathds{E}(e^{zY}-1)\right )^{-t}=\sum_{m=0}^\infty \binom{-t}{m}\left ( \mathds{E}e^{zY}-1\right )^m \\
         & =\sum_{m=0}^\infty \binom{-t}{m}m! \sum_{n=m}^\infty \dfrac{S_Y(n,m)}{n!}z^n=\sum_{n=0}^\infty \dfrac{z^n}{n!}\sum_{m=0}^n \binom{-t}{m}m! S_Y(n,m).
     \end{split}
\end{equation*}
This, in conjunction with \eqref{eq1.6}, shows the first equality in \eqref{eq1.8}.

By \eqref{eq2.8star}, we have
\begin{equation*}
    \begin{split}
        & \sum_{m=0}^n \binom{-t}{m}m! S_Y(n,m)=\sum_{m=0}^n \binom{-t}{m}\sum_{k=0}^m \binom{m}{k}(-1)^{m-k}\mathds{E}S_k^n\\
         & =\sum_{k=0}^n \mathds{E}S_k^n \sum_{m=k}^n \binom{-t}{m}\binom{m}{k}(-1)^{m-k}=\sum_{k=0}^n \binom{-t}{k}\binom{n+t}{n-k}\mathds{E}S_k^n,
     \end{split}
\end{equation*}
where the last equality can be checked by using the elementary combinatorial identity
\begin{equation*}
    \sum_{i=0}^p\binom{s+i}{i}=\binom{s+1+p}{p},\quad p\in \mathds{N}_0,\quad s\in \mathds{R}.
\end{equation*}
This shows the second equality in \eqref{eq1.8} and completes the proof.\hfill$\square$

\section{Examples}\label{s4}

In this section, we obtain explicit expression for various kinds of generalizations of the classical Bernoulli and Apostol-Euler polynomials in a unified way. This is possible because all of these polynomials belong to the class $\mathcal{E}_t(Y)$ for a suitable choice of the random variable $Y$ in each case.

Let $t\in \mathds{R}$ and $m\in \mathds{N}$. We consider the polynomials $B(t,m;x)$ defined by means of the generating function
\begin{equation}\label{eq3.24}
G(B(t,m;x),z)=\left ( \dfrac{z^m/m!}{e^z-\displaystyle\sum_{k=0}^{m-1}z^k/k!}\right )^t e^{xz}.
\end{equation}
For $t=1$, such polynomials have been considered in Bretti \textit{et al.} \cite{BreNatRic2004} and Khan and Riyasat \cite{KhaRiy2015}, among others. Important special cases are the generalized Bernoulli polynomials $B(t;x)$ of order $t$ and the classical Bernoulli polynomials $B(x)$ respectively given by
\begin{equation*}
B(t;x)=B(t,1;x),\qquad B(x)=B(1;x).
\end{equation*}

Let $\beta_m$ be a random variable having the beta density
\begin{equation}\label{eq3.25}
\rho_m(\theta)=m(1-\theta)^{m-1},\quad 0\leq \theta\leq 1.
\end{equation}
Observe that $\beta_1$ is uniformly distributed on $[0,1]$. Using the finite Taylor expansion with remainder in integral form (see also \cite[Lemma 7]{AdeBusQue2015}, we have
\begin{equation*}
e^z-\sum_{k=0}^{m-1} \dfrac{z^k}{k!}=\dfrac{z^m}{m!}\mathds{E}e^{z\beta_m}.
\end{equation*}
This implies that \eqref{eq3.24} can be rewritten as
\begin{equation}\label{eq3.26}
G(B(t,m;x),z)=\left (\mathds{E}e^{z\beta_m} \right )^{-t}e^{xz},
\end{equation}
which, thanks to \eqref{eq1.6}, implies that $B(t,m;x)\in \mathcal{E}_t(Y)$ with associated random variable $Y=\beta_m$. By \eqref{eq1.10} and \eqref{eq3.25}, the $j$th moment of the random variable $Y=\beta_m$ is given by
\begin{equation}\label{eq3.27}
  \mu_j=\mathds{E}\left (\beta_m \right )^j=m\int_0^1 \theta^j (1-\theta)^{m-1}\, d\theta= m\beta (j+1,m)=\dfrac{1}{\binom{m+j}{m}},
\end{equation}
for any $j\in \mathds{N}_0$. As follows from \eqref{eq1.11} and \eqref{eq3.27}, we have in the case at hand
\begin{equation}\label{eq4.6star}
    C(m,n,k):=\mathds{E}S_k^n=n!(m!)^k \sum_{j_1+\cdots +j_k=n}\dfrac{1}{(m+j_1)!\cdots (m+j_k)!}.
\end{equation}

As an immediate consequence of the second equality in \eqref{eq1.8}, we give the following.

\begin{corollary}\label{cor4.1}
Let $t\in \mathds{R}$ and $m\in \mathds{N}$. Then,
\begin{equation*}
    B_n(t,m;0)=\sum_{k=0}^n \binom{-t}{k}\binom{n+t}{n-k}C(m,n,k),
\end{equation*}
where $C(m,n,k)$ is defined in \eqref{eq4.6star}.
\end{corollary}

As follows from Theorem~\ref{th1}, a formula for $B_n(t,m;0)$ in terms of the generalized Stirling numbers $S_{\beta_m}(n,k)$ could be given. However, such a formula seems to be more involved than that in Corollary~\ref{cor4.1}.

For $m=1$, that is, for the generalized Bernoulli polynomials $B(t;x)$ of order $t$, we give the following representation.

\begin{corollary}\label{cor4.2}
Let $t\in \mathds{R}$. Then,
\begin{equation*}
    B_n(t;0)=\sum_{k=0}^n \binom{-t}{k}\binom{n+t}{n-k}\dfrac{S(n+k,k)}{\binom{n+k}{n}}.
\end{equation*}
\end{corollary}

\begin{proof}
The result follows from the second equality in \eqref{eq1.8} by taking into account the following fact. Since $Y=\beta_1$ has the uniform distribution on $[0,1]$, we have from \eqref{eq1.7} and \eqref{eq2.12star}
\begin{equation*}
    \mathds{E}S_k^n=\dfrac{S(n+k,k)}{\binom{n+k}{n}}.
\end{equation*}
The proof is complete.
\end{proof}

Corollary~\ref{cor4.2} was already obtained by Todorov \cite[Eq. (3)]{Tod1985} using different techniques (see also Srivastava and Todorov \cite{SriTod1988}). A different representation for the Bernoulli numbers in terms of the Stirling numbers of the second kind has been provided by Guo and Qi \cite[Theorem 3.1]{GuoQi2014}. For $t\in \mathds{N}$, a similar formula to that in Corollary~\ref{cor4.2} has been obtained by Kim \textit{et al.} \cite{KimKimDolRim2013} using umbral calculus techniques. For generalized Apostol-Bernoulli polynomials of integer order, we refer the read to Luo and Srivastava \cite{LuoSri2005}.

Let $t\in \mathds{R}$ and $0\leq \beta \leq 1$. We consider the generalized Apostol-Euler polynomials $E(t,\beta;x)$ of real order $t$ defined via its generating function as
\begin{equation}\label{eq3.33}
G(E(t,\beta;x),z)=\dfrac{e^{xz}}{(1+\beta (e^z-1))^t}.
\end{equation}
The Apostol-Euler polynomials and the classical Euler polynomials are defined, respectively, as
\begin{equation*}
E(\beta;x)=E(1,\beta;x),\qquad E(x)=E(1/2;x).
\end{equation*}
Up to a constant factor, the polynomials $E(t,\beta;x)$ were introduced by Luo \cite{Luo2006} and subsequently investigated by many authors (see, for instance, Wang \textit{et al.} \cite{WanJiaWan2008}, Luo \cite{Luo2010}, and Tremblay \textit{et al.} \cite{TreGabFug2011}, among others).

Let $X(\beta)$ be a random variable having the Bernoulli law
\begin{equation}\label{eq4.13star}
    P(X(\beta)=1)=1-P(X(\beta)=0)=\beta.
\end{equation}
Using  \eqref{eq4.13star}, we rewrite \eqref{eq3.33} as
\begin{equation}\label{eq4.8star}
    G(E(t,\beta;x),z)=\dfrac{e^{xz}}{\left ( \mathds{E}e^{zX(\beta)}\right )^t},
\end{equation}
thus showing that $E(t,\beta;x)\in \mathcal{E}_t(X(\beta))$.

\begin{corollary}\label{cor4.3}
Let $t\in \mathds{R}$ and $0\leq \beta \leq 1$. Then,
\begin{equation*}
    E_n(t,\beta;0)=\sum_{m=0}^n\binom{-t}{m}\beta^m m! S(n,m).
\end{equation*}
\end{corollary}

\begin{proof}
Since
\begin{equation*}
    \mathds{E}e^{zX(\beta)}-1=\beta(e^z-1),
\end{equation*}
formula \eqref{eq2.9star} immediately yields
\begin{equation*}
    S_{X(\beta)}(n,m)=\beta^m S(n,m),\quad n,m\in \mathds{N}_0,\quad m\leq n.
\end{equation*}
Thus, the result follows from the first equality in \eqref{eq1.8}.
\end{proof}

When $t$ is a positive integer, Corollary~\ref{cor4.3} has also been shown by Kim \textit{et al.} \cite{KimKimDolRim2013} by using umbral calculus. A similar, but more involved formula than that in Corollary~\ref{cor4.3}, was obtained by Luo \cite{Luo2006}.

Let $t\in \mathds{R}$ and $0\leq \beta \leq 1$. Assume that $Y$ is a random variable uniformly distributed on $[0,1]$ and independent of $X(\beta)$, as defined in \eqref{eq4.13star}. Observe that
\begin{equation}\label{eq4.16star}
    \mathds{E}e^{zX(\beta)Y}=\beta \mathds{E}e^{zY}+1 -\beta=\beta \dfrac{e^z-1}{z}+1-\beta.
\end{equation}
We finally consider the polynomials $B^*(t,\beta;x)$ defined by means of the generating function
\begin{equation}\label{eq4.9star}
    G(B^*(t,\beta;x),z)=\dfrac{e^{xz}}{\left ( \mathds{E}e^{zX(\beta)Y}\right )^t}=\left ( \dfrac{z}{\beta (e^z-1)+(1-\beta)z}\right )^t e^{xz}.
\end{equation}
Note that $B^*(t,\beta;x)\in \mathcal{E}_t(X(\beta)Y)$ and that $B^*(t,1;x)=B(t,x)$. As far as we know, the polynomials $B^*(t,\beta;x)$ are introduced here for the first time. For such polynomials, we give the following explicit expression.

\begin{corollary}\label{cor4.4}
Let $t\in \mathds{R}$ and $0\leq \beta \leq 1$. Then,
\begin{equation*}
    B^*(t,\beta;0)=\sum_{m=0}^n \binom{-t}{m}\beta^m \sum_{k=0}^m \binom{m}{k}(-1)^{m-k}\dfrac{S(n+k,k)}{\binom{n+k}{n}}.
\end{equation*}
\end{corollary}

\begin{proof}
As follows from \eqref{eq4.16star}, we have
\begin{equation*}
    \mathds{E}\left ( e^{zX(\beta)Y}-1\right )=\beta\mathds{E}\left ( e^{zY}-1\right ).
\end{equation*}
Thus, formula \eqref{eq2.9star} gives us
\begin{equation}\label{eq4.14star}
    S_{X(\beta)Y}(n,m)=\beta^m S_Y(n,m),\quad n,m\in \mathds{N}_0,\quad m\leq n.
\end{equation}
On the other hand, since $Y$ is uniformly distributed on $[0,1]$, we have from \eqref{eq2.8star} and \eqref{eq2.12star}
\begin{equation}\label{eq4.15star}
\begin{split}
    m! S_Y(n,m)&=\sum_{k=0}^m \binom{m}{k}(-1)^{m-k}\mathds{E}S_k^n\\
    &=\sum_{k=0}^m \binom{m}{k}(-1)^{m-k}\dfrac{S(n+k,k)}{\binom{n+k}{n}}.
    \end{split}
\end{equation}
Therefore, the conclusion follows from \eqref{eq4.14star}, \eqref{eq4.15star}, and the first equality in \eqref{eq1.8}
\end{proof}



\section*{References}

\bibliographystyle{elsarticle-num}
\bibliography{AdellLekuonaArXiv2017}

\begin{thebibliography}{10}
\expandafter\ifx\csname url\endcsname\relax
  \def\url#1{\texttt{#1}}\fi
\expandafter\ifx\csname urlprefix\endcsname\relax\def\urlprefix{URL }\fi
\expandafter\ifx\csname href\endcsname\relax
  \def\href#1#2{#2} \def\path#1{#1}\fi

\bibitem{ErdMagObeTri1955}
A.~Erd\'elyi, W.~Magnus, F.~Oberhettinger, F.~G. Tricomi, Higher
  {T}ranscendental {F}unctions. {V}ol. {III}, McGraw-Hill Book Company, Inc.,
  New York-Toronto-London, 1955, based, in part, on notes left by Harry
  Bateman.

\bibitem{GraKnuPat1989}
R.~L. Graham, D.~E. Knuth, O.~Patashnik, {C}oncrete {M}athematics,
  Addison-Wesley Publishing Company, Advanced Book Program, Reading, MA, 1989,
  a foundation for computer science.

\bibitem{AbrSte1992}
M.~Abramowitz, I.~A. Stegun, {H}andbook of {M}athematical {F}unctions with
  {F}ormulas, {G}raphs, and {M}athematical {T}ables, Dover, New York, 1992.

\bibitem{OlvLozBoiCla2010}
F.~W.~J. Olver, D.~W. Lozier, R.~F. Boisvert, C.~W. Clark (Eds.), N{IST}
  {H}andbook of {M}athematical {F}unctions, U.S. Department of Commerce,
  National Institute of Standards and Technology, Washington, DC; Cambridge
  University Press, Cambridge, 2010, with 1 CD-ROM (Windows, Macintosh and
  UNIX).

\bibitem{Tod1985}
P.~G. Todorov, Une formule simple explicite des nombres de {B}ernoulli
  g\'en\'eralis\'es, C. R. Acad. Sci. Paris S\'er. I Math. 301~(13) (1985)
  665--666.

\bibitem{SriTod1988}
H.~M. Srivastava, P.~G. Todorov,
  \href{http://dx.doi.org/10.1016/0022-247X(88)90326-5}{An explicit formula for
  the generalized {B}ernoulli polynomials}, J. Math. Anal. Appl. 130~(2) (1988)
  509--513.
\newblock \href {http://dx.doi.org/10.1016/0022-247X(88)90326-5}
  {\path{doi:10.1016/0022-247X(88)90326-5}}.
\newline\urlprefix\url{http://dx.doi.org/10.1016/0022-247X(88)90326-5}

\bibitem{GuoQi2014}
B.-N. Guo, F.~Qi, \href{http://dx.doi.org/10.1016/j.cam.2013.06.020}{Some
  identities and an explicit formula for {B}ernoulli and {S}tirling numbers},
  J. Comput. Appl. Math. 255 (2014) 568--579.
\newblock \href {http://dx.doi.org/10.1016/j.cam.2013.06.020}
  {\path{doi:10.1016/j.cam.2013.06.020}}.
\newline\urlprefix\url{http://dx.doi.org/10.1016/j.cam.2013.06.020}

\bibitem{Luo2006}
Q.-M. Luo, Apostol-{E}uler polynomials of higher order and {G}aussian
  hypergeometric functions, Taiwanese J. Math. 10~(4) (2006) 917--925.

\bibitem{LuoSri2011}
Q.-M. Luo, H.~M. Srivastava,
  \href{http://dx.doi.org/10.1016/j.amc.2010.12.048}{Some generalizations of
  the {A}postol-{G}enocchi polynomials and the {S}tirling numbers of the second
  kind}, Appl. Math. Comput. 217~(12) (2011) 5702--5728.
\newblock \href {http://dx.doi.org/10.1016/j.amc.2010.12.048}
  {\path{doi:10.1016/j.amc.2010.12.048}}.
\newline\urlprefix\url{http://dx.doi.org/10.1016/j.amc.2010.12.048}

\bibitem{Wuy2013}
Wuyungaowa, \href{http://dx.doi.org/10.1186/1029-242X-2013-422}{Identities on
  products of {G}enocchi numbers}, J. Inequal. Appl. (2013) 2013:422, 10\href
  {http://dx.doi.org/10.1186/1029-242X-2013-422}
  {\path{doi:10.1186/1029-242X-2013-422}}.
\newline\urlprefix\url{http://dx.doi.org/10.1186/1029-242X-2013-422}

\bibitem{Ta2015}
B.~Q. Ta, \href{http://dx.doi.org/10.1016/j.exmath.2014.07.003}{Probabilistic
  approach to {A}ppell polynomials}, Expo. Math. 33~(3) (2015) 269--294.
\newblock \href {http://dx.doi.org/10.1016/j.exmath.2014.07.003}
  {\path{doi:10.1016/j.exmath.2014.07.003}}.
\newline\urlprefix\url{http://dx.doi.org/10.1016/j.exmath.2014.07.003}

\bibitem{HsuShi1998}
L.~C. Hsu, P.~J.-S. Shiue, \href{http://dx.doi.org/10.1006/aama.1998.0586}{A
  unified approach to generalized {S}tirling numbers}, Adv. in Appl. Math.
  20~(3) (1998) 366--384.
\newblock \href {http://dx.doi.org/10.1006/aama.1998.0586}
  {\path{doi:10.1006/aama.1998.0586}}.
\newline\urlprefix\url{http://dx.doi.org/10.1006/aama.1998.0586}

\bibitem{Mez2010}
I.~Mez\H{o}, \href{http://dx.doi.org/10.1007/s00025-010-0039-z}{A new formula
  for the {B}ernoulli polynomials}, Results Math. 58~(3-4) (2010) 329--335.
\newblock \href {http://dx.doi.org/10.1007/s00025-010-0039-z}
  {\path{doi:10.1007/s00025-010-0039-z}}.
\newline\urlprefix\url{http://dx.doi.org/10.1007/s00025-010-0039-z}

\bibitem{AdeLek2017.2}
J.~A. Adell, A.~Lekuona,
  \href{http://dx.doi.org/10.1016/j.jmaa.2017.06.077}{Binomial convolution and
  transformations of {A}ppell polynomials}, J. Math. Anal. Appl. 456~(1) (2017)
  16--33.
\newblock \href {http://dx.doi.org/10.1016/j.jmaa.2017.06.077}
  {\path{doi:10.1016/j.jmaa.2017.06.077}}.
\newline\urlprefix\url{http://dx.doi.org/10.1016/j.jmaa.2017.06.077}

\bibitem{MroRajWas2017}
J.~Mrowiec, T.~Rajba, S.~W\polhk{a}sowicz,
  \href{http://dx.doi.org/10.1016/j.jmaa.2017.01.083}{On the classes of
  higher-order {J}ensen-convex functions and {W}right-convex functions, {II}},
  J. Math. Anal. Appl. 450~(2) (2017) 1144--1147.
\newblock \href {http://dx.doi.org/10.1016/j.jmaa.2017.01.083}
  {\path{doi:10.1016/j.jmaa.2017.01.083}}.
\newline\urlprefix\url{http://dx.doi.org/10.1016/j.jmaa.2017.01.083}

\bibitem{DilVig2016}
K.~Dilcher, C.~Vignat,
  \href{http://dx.doi.org/10.1016/j.jmaa.2015.11.006}{General convolution
  identities for {B}ernoulli and {E}uler polynomials}, J. Math. Anal. Appl.
  435~(2) (2016) 1478--1498.
\newblock \href {http://dx.doi.org/10.1016/j.jmaa.2015.11.006}
  {\path{doi:10.1016/j.jmaa.2015.11.006}}.
\newline\urlprefix\url{http://dx.doi.org/10.1016/j.jmaa.2015.11.006}

\bibitem{Sun2005}
P.~Sun, \href{http://dx.doi.org/10.1007/s10114-005-0631-4}{Product of uniform
  distribution and {S}tirling numbers of the first kind}, Acta Math. Sin.
  (Engl. Ser.) 21~(6) (2005) 1435--1442.
\newblock \href {http://dx.doi.org/10.1007/s10114-005-0631-4}
  {\path{doi:10.1007/s10114-005-0631-4}}.
\newline\urlprefix\url{http://dx.doi.org/10.1007/s10114-005-0631-4}

\bibitem{BreNatRic2004}
G.~Bretti, P.~Natalini, P.~E. Ricci,
  \href{http://dx.doi.org/10.1155/S1085337504306263}{Generalizations of the
  {B}ernoulli and {A}ppell polynomials}, Abstr. Appl. Anal.~(7) (2004)
  613--623.
\newblock \href {http://dx.doi.org/10.1155/S1085337504306263}
  {\path{doi:10.1155/S1085337504306263}}.
\newline\urlprefix\url{http://dx.doi.org/10.1155/S1085337504306263}

\bibitem{KhaRiy2015}
S.~Khan, M.~Riyasat, \href{http://dx.doi.org/10.1016/j.jmaa.2014.07.044}{A
  determinantal approach to {S}heffer-{A}ppell polynomials via monomiality
  principle}, J. Math. Anal. Appl. 421~(1) (2015) 806--829.
\newblock \href {http://dx.doi.org/10.1016/j.jmaa.2014.07.044}
  {\path{doi:10.1016/j.jmaa.2014.07.044}}.
\newline\urlprefix\url{http://dx.doi.org/10.1016/j.jmaa.2014.07.044}

\bibitem{AdeBusQue2015}
J.~A. Adell, J.~Bustamante, J.~M. Quesada,
  \href{http://dx.doi.org/10.1016/j.jmaa.2015.05.020}{Estimates for the moments
  of {B}ernstein polynomials}, J. Math. Anal. Appl. 432~(1) (2015) 114--128.
\newblock \href {http://dx.doi.org/10.1016/j.jmaa.2015.05.020}
  {\path{doi:10.1016/j.jmaa.2015.05.020}}.
\newline\urlprefix\url{http://dx.doi.org/10.1016/j.jmaa.2015.05.020}

\bibitem{KimKimDolRim2013}
D.~S. Kim, T.~Kim, D.~V. Dolgy, S.-H. Rim,
  \href{http://dx.doi.org/10.1186/1687-1847-2013-73}{Some new identities of
  {B}ernoulli, {E}uler and {H}ermite polynomials arising from umbral calculus},
  Adv. Difference Equ. (2013) 2013:73, 11\href
  {http://dx.doi.org/10.1186/1687-1847-2013-73}
  {\path{doi:10.1186/1687-1847-2013-73}}.
\newline\urlprefix\url{http://dx.doi.org/10.1186/1687-1847-2013-73}

\bibitem{LuoSri2005}
Q.-M. Luo, H.~M. Srivastava,
  \href{http://dx.doi.org/10.1016/j.jmaa.2005.01.020}{Some generalizations of
  the {A}postol-{B}ernoulli and {A}postol-{E}uler polynomials}, J. Math. Anal.
  Appl. 308~(1) (2005) 290--302.
\newblock \href {http://dx.doi.org/10.1016/j.jmaa.2005.01.020}
  {\path{doi:10.1016/j.jmaa.2005.01.020}}.
\newline\urlprefix\url{http://dx.doi.org/10.1016/j.jmaa.2005.01.020}

\bibitem{WanJiaWan2008}
W.~Wang, C.~Jia, T.~Wang,
  \href{http://dx.doi.org/10.1016/j.camwa.2007.06.021}{Some results on the
  {A}postol-{B}ernoulli and {A}postol-{E}uler polynomials}, Comput. Math. Appl.
  55~(6) (2008) 1322--1332.
\newblock \href {http://dx.doi.org/10.1016/j.camwa.2007.06.021}
  {\path{doi:10.1016/j.camwa.2007.06.021}}.
\newline\urlprefix\url{http://dx.doi.org/10.1016/j.camwa.2007.06.021}

\bibitem{Luo2010}
Q.-M. Luo, An explicit relationship between the generalized
  {A}postol-{B}ernoulli and {A}postol-{E}uler polynomials associated with
  {$\lambda$}-{S}tirling numbers of the second kind, Houston J. Math. 36~(4)
  (2010) 1159--1171.

\bibitem{TreGabFug2011}
R.~Tremblay, S.~Gaboury, B.-J. Fug\`ere,
  \href{http://dx.doi.org/10.1016/j.aml.2011.05.012}{A new class of generalized
  {A}postol-{B}ernoulli polynomials and some analogues of the
  {S}rivastava-{P}int\'er addition theorem}, Appl. Math. Lett. 24~(11) (2011)
  1888--1893.
\newblock \href {http://dx.doi.org/10.1016/j.aml.2011.05.012}
  {\path{doi:10.1016/j.aml.2011.05.012}}.
\newline\urlprefix\url{http://dx.doi.org/10.1016/j.aml.2011.05.012}

\end{thebibliography}

\end{document}